\documentclass[11pt]{amsart}
\usepackage{ryanmacros}

\usepackage{tikz}
\usetikzlibrary{matrix,arrows,backgrounds,shapes.misc,shapes.geometric,patterns,calc,positioning}

\usepackage[colorlinks=true, pdfstartview=FitV, linkcolor=blue, citecolor=blue, urlcolor=blue]{hyperref}

\usepackage{fullpage}

\DeclareMathOperator{\Aut}{Aut}

\DeclareMathOperator{\gr}{gr}
\newcommand{\keyw}[1]{\emph{#1}}
\newcommand{\za}{\alpha}
\newcommand{\zb}{\beta}
\newcommand{\zg}{\gamma}

\newcommand{\FF}{\mathbb{F}}

\newcommand{\dd}{\underline{d}}
\newcommand{\usg}{\widetilde{S_g}}

\newcommand{\twobytwo}[4]{\left(\begin{matrix}
					 #1 & #2 \\ #3 & #4
					 \end{matrix}\right)}

\begin{document}
\title{Tree modules and counting polynomials}
\author{Ryan Kinser}
\address{Department of Mathematics, Northeastern University, Boston, MA}
\email{r.kinser@neu.edu}

\begin{abstract}
We give a formula for counting tree modules for the quiver $S_g$ with $g$ loops and one vertex in terms of tree modules on its universal cover.  This formula, along with work of Helleloid and Rodriguez-Villegas, is used to show that the number of $d$-dimensional tree modules for $S_g$ is polynomial in $g$ with the same degree and leading coefficient as the counting polynomial $A_{S_g}(d, q)$ for absolutely indecomposables over $\mathbb{F}_q$, evaluated at $q=1$.
\end{abstract}
\maketitle


\section{Introduction}\label{sect:intro}

The goal of the paper is to establish a connection between two apparently unrelated objects associated to a quiver: its tree modules, and the polynomials introduced by Kac which count certain representations over finite fields.  The background on these is reviewed in the following section.  Here we focus on the ``$g$-loop quiver'' $S_g$ for simplicity; the work may generalize to other quivers but the calculations of the main theorem are not as straightforward.  Denote by $A_{S_g}(d, q)$ the polynomial counting isomorphism classes of  $d$-dimensional absolutely indecomposable representations of $S_g$ over $\mathbb{F}_q$.

Motivated by the work of Hausel and Rodriguez-Villegas \cite{MR2453601}, Helleloid and Rodriguez-Villegas obtained a formula for the evaluation of $A_{S_g}(d, q)$ at $q=1$ \cite{MR2543630}; in particular, they found that it is a polynomial in $g$ whose leading term can be calculated by enumeration of labeled trees.  
It is natural to try to connect this evaluation at $q=1$ back to the representation theory of $S_g$.

While there is not currently any simple interpretation of $A_{S_g}(d,q)$ at $q=1$, tree modules are representations whose structure is completely presented by a finite, labeled tree, and thus can be regarded as combinatorial in some sense.  (In this paper, we use the term ``tree modules'' as in \cite{Ringel:1998gf}, which is less restrictive than, say, \cite{Crawley-Boevey:1989rr}.) Our main theorem utilizes covering theory of quivers to relate the number of $d$-dimensional tree modules for $S_g$ to $A_{S_g}(d, 1)$.

\begin{theoremnonum}
The number of tree modules of dimension $d$ for $S_g$ is a polynomial in $g$ with the same degree $d-1$ and same leading coefficient as $A_{S_g}(d, 1)$, namely
\begin{equation}
\sum_{Q \in \mathscr{Q}_d} \frac{1}{\#\Aut Q}= \frac{2^{d-1}d^{d-2}}{d!} .
\end{equation}
Here, $\mathscr{Q}_d$ is the set of tree quivers with $d$ vertices.
\end{theoremnonum}

In the final section, we give some examples that motivated this result, and pose questions.  Whenever a quasiprojective variety admits a counting polynomial, evaluation at $q=1$  gives its Euler characteristic in singular cohomology with compact support.  If this variety additionally admits an algebraic cell decomposition, this counts the number of cells.  In many examples, we have observed that the number of tree modules of a given dimension is greater than or equal to the evaluation at the relevant counting polynomial at $q=1$. This leads us to speculate on how our result may relate to a conjecture of Kac on cell decompositions of varieties of isomorphism classes of indecomposable representations of quivers.

\section{Background}\label{sect:back}
\subsection{Quiver representations}
We freely use the basic notions of quivers and their representations; in particular we interchangeably use the terminology of quiver representations and modules over the path algebra, since the categories are equivalent.  We only work with finite-dimensional representations throughout this paper.  The books \cite{assemetal,ARSartinalgebras} are standard references for this background.
For $g \in \Z_{\geq 0}$, we denote by $S_g$ the quiver with one vertex and $g$ loops, so that its path algebra $KS_g$ over a field $K$ is the free (associative) $K$-algebra on $g$ generators. 

A \keyw{quiver} $Q$ is written as a pair $Q=(Q_0, Q_1)$ where $Q_0$ is the vertex set of $Q$ and $Q_1$ the arrow set of $Q$.  For an arrow $\za \in Q_1$, we denote by $t\za$ the tail (starting point) of $\za$ and $h\za$ the head (ending point) of $\za$.  A \keyw{representation of $Q$} is given by associating a vector space $V_x$ to each $x\in Q_0$ and a linear map $V_\za \colon V_{t\za} \to V_{h\za}$ to each $\za \in Q_1$.

A \keyw{morphism between quivers} is a morphism of directed graphs: it sends vertices to vertices and arrows to arrows, preserving the head and tail of each arrow.  
More precisely, $f\colon Q' \to Q$ is given by maps $f_0\colon Q'_0 \to Q_0$ and $f_1\colon Q'_1 \to Q_1$ which satisfy
\[
f_0(t\za) = tf_1(\za) \qquad \text{and} \qquad f_0(h\za) = hf_1(\alpha)
\]
for all $\za \in Q_1$.

Fixing a base quiver $Q$, we say that a \keyw{quiver over $Q$} is a pair $(Q', f)$ consisting of another quiver $Q'$ and a structure map $f\colon Q' \to Q$.  We allow $Q'$ to have infinitely many vertices and arrows, but we still consider finite dimensional representations (in particular, our representations still have finite support).  The map $f$ induces functors
\[
f_* \colon \rep(Q') \to \rep(Q) \qquad \text{and} \qquad f^*\colon \rep(Q) \to \rep(Q')
\]
which we describe here.
To simplify the notation, we consider the maps $V_{\za}$ of a representation $V$ to be defined on the total vector space $\bigoplus_{x \in Q_0} V_x$ by taking $V_\za (v) = 0$ for $v \in V_y$, when $y \neq t\za$.
The \keyw{pullback} $f^* (W) \in \rep(Q')$ of a representation $W \in \repq$ along a morphism of quivers $f\colon Q' \to Q$ is given by
\[
f^*(W)_x := W_{f(x)} \qquad x \in Q'_0 \qquad \qquad f^*(W)_\za := W_{f(a)} \qquad \za \in Q'_1,
\]
and the \keyw{pushforward} $f_* (V) \in \repq$ of a representation $V \in \rep(Q')$ is given by
\[
f_*(V)_x := \bigoplus_{y \in f^{-1}(x)} V_y \qquad x \in Q_0 \qquad \qquad f_*(V)_\za := \sum_{\zb \in f^{-1}(\za)} V_\zb \qquad \za \in Q_1 .
\]

As a side remark, we mention that pushforward and pullback are not generally adjoint functors, but this can happen in special cases (see for example \cite[\S 2.1]{kinserrootedtrees}).

\subsection{Tree modules}
When $T$ is a tree quiver, that is, a quiver whose underlying graph is a tree, it has a unique (up to isomorphism) indecomposable representation $\mathbb{I}_T$ which is one-dimensional at each vertex; we may take the map over each arrow in $\mathbb{I}_T$ to be the identity.  Given an indecomposable representation $M$ of any quiver $Q$, we say that $M$ is a \keyw{tree module} if there exists $T \xto{f} Q$, with $T$ a tree quiver, such that $M \simeq f_*(\mathbb{I}_T)$.   In this case we say that $(T, f)$ is a \keyw{structure quiver} for $M$.
In general, the quiver $T$ is not unique nor is its underlying graph (see \cite[Rmk. 2]{Ringel:1998gf} for an example with $Q$ of type $D_4$). 

A more concrete formulation is that an indecomposable $M$ is a tree module if and only if it admits a basis $\mathcal{M}=\{m_1, m_2, \dotsc, m_d\}$ such that the collection of matrices presenting $M$ as a representation of $Q$ in this basis consist of ones and zeros, with precisely $d-1$ nonzero entries. 
Note that if $M$ admits a basis for which the corresponding matrices have strictly fewer than $d-1$ nonzero entries, then $M$ is decomposable (see, for example, \cite[Property 1]{Ringel:1998gf}).  So we can think of tree modules as indecomposable representations which can be presented as sparsely as possible.  The equivalence of the two definitions comes from identifying the basis vectors $\mathcal{M}$ with the vertices of $T$.


\subsection{Counting polynomials}\label{sect:countpoly}
We denote by $\FF_q$ the finite field with $q$ elements, and recall that a quiver representation over $\FF_q$ is \keyw{absolutely indecomposable} if it is indecomposable after extension of scalars to an algebraic closure $\overline{\FF_q}$.
For a fixed quiver $Q$ and dimension vector $\dd$, Kac proved that there is a polynomial with integer coefficients
\[
A_Q(\dd, q) = q^n + c_{n-1}q^{n-1} + \cdots + c_1 q + c_0 \in \Z[q]
\]
which counts the number of isomorphism classes of absolutely indecomposable $\FF_q$-representations of $Q$ of dimension vector $\dd$ \cite[Prop. 1.15]{MR718127}.
(We remark once and for all that there may be finitely many ``bad'' characteristics for which this is not actually true.)
Kac made several conjectures about the coefficients of these polynomials.  The first interprets the constant term $c_0$ as the dimension of the $\dd$-root space in the Kac-Moody algebra determined by $Q$ (when $Q$ has no loops), and the second conjecture predicts that all $c_i \geq 0$ (for any $Q$).
Both conjectures were proven for indivisible dimension vectors by Crawley-Boevey and van den Bergh \cite{MR2038196}, and the first conjecture was later proven by Hausel \cite{MR2651380} in full generality.
Kac proposes a third, even more bold conjecture, which roughly says that the number $c_i$ counts $i$-dimensional algebraic cells in a variety whose points parametrize isomorphism classes of indecomposable representations of $Q$ of dimension vector $\dd$ (see Section \ref{sect:examples}).

\subsection{Counting polynomials at $q=1$}
Effectively computing $A_Q(\dd, q)$ seems to be a difficult problem in general.  LeBruyn's work \cite{LeBruyn1986d} on the case $Q=S_g$  provided many examples that helped in exploring ideas of this paper; later the author became aware of a combinatorial formula by Jiuzhao Hua \cite{MR1752774} which is quite reasonable to implement for computer calculations.

Helleloid and Rodriguez-Villegas used Hua's formula to get some information on the evaluation $A_Q(\dd,q=1)$, for any quiver $Q$.  We just use a special case of their results when $Q=S_g$, as summarized in the following theorem (see Theorem 6.3 in \cite{MR2543630} and the paragraph following it).

\begin{theorem}[Helleloid and Rodriguez-Villegas]
For a fixed dimension $d$, we have that $A_{S_g}(d, q=1)$ is a polynomial in $g$ with leading term
\begin{equation}\label{eq:RVlead}
\frac{2^{d-1}\#\mathscr{T}_d}{d!}g^{d-1},
\end{equation}
where $\mathscr{T}_d$ is the set of (connected) trees on $d$ labeled vertices.
\end{theorem}

We remark that $\#\mathscr{T}_d = d^{d-2}$ by a theorem of Cayley \cite{Cayley:1889zl}.

\section{Counting Tree Modules}\label{sect:counting}

\subsection{The universal cover of $S_g$}
The universal cover of $S_g$ is a quiver $\usg$ with structure map $\pi \colon \usg \to S_g$ which is a topological universal cover when we regard the quivers as graphs with the standard topology.
The group of automorphisms of $\usg$ that respect $\pi$ is isomorphic to the free group $F_g$ generated by the arrows $\za_1, \dotsc, \za_g$ of $S_g$.
 
In more detail, we take the set of vertices of $\usg$ to be the set of elements of $F_g$,
with a labeled arrow $v \xto{\za_i} w$ in $\usg$ whenever $w = \za_i v$ in $F_g$.  The covering map $\pi\colon \usg \to S_g$ sends every vertex to the unique vertex of $S_g$, and it sends each arrow of $\usg$ to its label in $S_g$.   The group $F_g$ acts naturally by left multiplication on the vertices of $\usg$ which then uniquely defines the action on the arrows of $\usg$ since $\usg$ is a tree.  It follows from the definition that these are deck transformations; in other words, the action of $F_g$ leaves the labels of the arrows in $\usg$ fixed.
The following lemma shows that every tree module for $S_g$ comes from a tree module on a subquiver of $\usg$.

\begin{lemma}\label{lem:sgfactor}
Let $M \in \rep(S_g)$ be a tree module.  Then there exists a tree module $N \in \rep(\usg)$ such that $M \simeq \pi_* (N)$.  \end{lemma}
\begin{proof}
Let $T \xto{f} S_g$ be a structure quiver for $M$.  First we claim that $f$ factors through $\usg$ in the category of quivers over $S_g$: that is, there exists a morphism of quivers $T \xto{f'} \usg$ making the following diagram commute.
\begin{equation}\label{eq:sgfactor}
\vcenter{\hbox{
\begin{tikzpicture}[>=latex,description/.style={fill=white,inner sep=2pt}] 
\matrix (m) [matrix of math nodes, row sep=3em, 
column sep=2.5em, text height=1.5ex, text depth=0.25ex] 
{  T &  & \usg  \\ 
 & S_g &  \\ }; 
\path[->,font=\scriptsize] 
(m-1-1) edge node[auto,swap] {$ f $} (m-2-2)
(m-1-3) edge node[auto] {$ \pi $} (m-2-2)
(m-1-1) edge node[auto] {$ f' $} (m-1-3);
$\circlearrowleft$;
\end{tikzpicture}}}
\end{equation}

To see this, fix an arbitrary vertex $t_0 \in T$, and set $f'(t_0) = e$ (the identity element of $F_g$, considered as a vertex of $\usg$).  Now since $T$ is a tree, there is a unique walk in $T$ from $t_0$ to any other vertex $t$.  We can express this walk as a word $\zb_1^{\pm 1} \zb_2^{\pm 1} \cdots \zb_k^{\pm 1}$ where each $\zb_i$ is an arrow of $T$, taking $\zb_i$ when the walk is in the same direction as the arrow and $\zb_i^{-1}$ when it goes against the direction of the arrow.  Now applying $f$ to this word, we get a word in the arrows of $S_g$ and their inverses, which is nothing other than an element of $F_g$, or a vertex of $\usg$.  This defines $f'(t)$ for each vertex $t \in T$, and the effect of $f'$ on the arrows of $T$ is then uniquely determined since $\usg$ is a tree.

Now from this diagram it is immediate that $ \pi_* (f'_* (\mathbb{I}_T)) = f_* (\mathbb{I}_T) \simeq M$, and so we can take $N = f'_*(\mathbb{I}_T)$.  Note that since pushforward commutes with direct sum, the assumption that $M$ is indecomposable assures that $N$ is as well, so it is a tree module for $\usg$.
\end{proof}

\begin{example}
The following diagram shows a tree $T$, with each arrow labeled by the value of the structure map $f$ on that arrow, and the associated map $f'$ which collapses the two vertices in the middle row.  The right hand side is regarded as a subquiver of $\widetilde{S_3}$ with the values of $\pi$ indicated for each arrow.
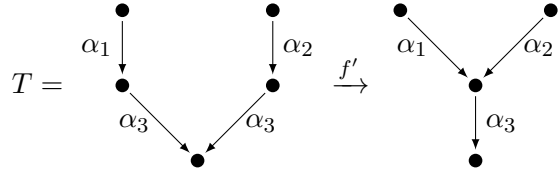
\[
T=\vcenter{\hbox{\begin{tikzpicture}[point/.style={shape=circle,fill=black,scale=.5pt,outer sep=3pt},>=latex]
   \node[point] (1) at (-1,2) {};
   \node[point] (2) at (1,2) {};
   \node[point] (3) at (-1,1) {};
   \node[point] (4) at (1,1) {};
  \node[point] (5) at (0,0) {};
  \path[->]
  	(1) edge node[left] {$\za_1$} (3) 
  	(2) edge node[right] {$\za_2$} (4)
  	(3) edge node[left] {$\za_3$} (5) 
  	(4) edge node[right] {$\za_3$} (5);
   \end{tikzpicture}}}
   \xto{ f'}
   \vcenter{\hbox{
   \begin{tikzpicture}[point/.style={shape=circle,fill=black,scale=.5pt,outer sep=3pt},>=latex]
   \node[point] (1) at (-1,2) {};
   \node[point] (2) at (1,2) {};
   \node[point] (3) at (0,1) {};
  \node[point] (5) at (0,0) {};
  \path[->]
  	(1) edge node[left] {$\za_1$} (3) 
  	(2) edge node[right] {$\za_2$} (3)
  	(3) edge node[right] {$\za_3$} (5);
   \end{tikzpicture}}  }
\]
The tree module $M \in \rep(S_3)$ determined by $T$ is isomorphic to $\pi_*(N)$, where $N$ is the indecomposable of maximal dimension supported on this $D_4$ type subquiver of $\widetilde{S_3}$.
\end{example}

The free group $F_g$ gives a grading to the free algebra $KS_g$ in the sense of \cite[\S 3]{MR704617}.  
The following proposition may be well known, but seems to only appear in the literature with hypotheses that are not satisfied in our case, so we sketch a proof in sufficient generality for our purposes.  
Here, $\Lambda$ is a finitely generated algebra over a field $K$ (not assumed to be algebraically closed), and $\Lambda$ is graded by a torsion free group $G$.  We denote by $\modules \Lambda$ (resp., $\gr\Lambda$) the category of finite $K$-dimensional $\Lambda$ modules  (resp., finite $K$-dimensional $G$-graded $\Lambda$ modules), 
 and $\pi_* \colon \gr\Lambda \to \modules\Lambda$ is the functor that forgets the grading.

The proof of the first statement below directly follows Gordon and Green's proof of \cite[Theorem 3.2]{MR659212} (see also \cite[Lemma 3.5]{gabrieluniversalcover}).

\begin{prop}\label{prop:covers}
For $N \in \gr\Lambda$, we have that $N$ is indecomposable if and only if $\pi_*(N)$ is indecomposable.
If $N$ is indecomposable and $N' \in \gr \Lambda$ is another representation such that $\pi_*(N) \simeq \pi_*(N')$, then there exists $h \in G$ such that $N' \simeq hN$.
\end{prop}
\begin{proof}  If $N$ is decomposable, then $\pi_*(N)$ is decomposable because $\pi_*$ commutes with direct sum.  The converse follows from the implications
\begin{align*}
N \text{ indecomposable in }\gr \Lambda &\overset{\text{(a)}}{\Longrightarrow} \End_{\gr\Lambda}(N) \text{ local} \overset{\text{(b)}}{\Longrightarrow} \\
\End_{\Lambda} (N) \text{ local} &\overset{\text{(c)}}{\Longrightarrow} \pi_*(N) \text{ indecomposable in }\modules \Lambda .
\end{align*}
Implications (a) and (c) follow from (for example) \cite[Ch.1~\S1]{Leuschke:2011fk}, using the facts that (i) idempotents split in $\modules \Lambda$ and $\gr\Lambda$, since these are both abelian categories, and that (ii) both $\End_{\gr\Lambda}(N)$ and $\End_{\Lambda} (N)$ are Artinian, since $N$ is finite dimensional over $K$.

Implication (b) is a direct generalization of \cite[Theorem~3.1]{MR659212} to our setting.  The assumption that $G$ is torsion free is used to assure 
a nonzero right annihilator 
for every element which can be written as a product of two elements outside of the identity degree.

The second statement follows from the vector space decomposition
\begin{equation}
\End_{\Lambda} (N) \cong \bigoplus_{h \in G} \Hom_{\gr \Lambda}(N, hN)
\end{equation}
(see the comments following \cite[Lemma~2.1]{MR659212}).
\end{proof}

Applying this to $\Lambda = KS_g$, we use \cite[Theorem 3.2]{MR704617} to conclude that
\[
\pi_* \colon \rep (\usg) \to \rep(S_g)
\]
has the same property: a representation $N \in \rep(\usg)$ is indecomposable if and only if $\pi_*(N)$ is indecomposable, and in this case, any $N' \in \rep(\usg)$ such that $\pi_*(N') \simeq \pi_*(N)$ satisfies $N' \simeq wN$ for some $w \in F_g$.

\subsection{Counting tree modules}
For any quiver $Q$, let $TM_{Q}(d)$ be the number of isomorphism classes of sincere tree modules for $Q$ of dimension $d$. (A representation of a quiver is \keyw{sincere} when it is nonzero at each vertex of the quiver.)  We denote by $\mathscr{Q}_n$ the set of tree quivers with $n$ (unlabeled) vertices, and set $\mathscr{Q} = \bigcup_{n\geq 1} \mathscr{Q}_n$.

Now we will count tree modules for $S_g$ by counting tree modules on subquivers of $\usg$.
The action of $F_g$ on $\usg$ induces an action on the set of subquivers of $\usg$.  For any tree quiver $Q$, let $[Q: \usg]$ be the number of $F_g$-orbits of subquivers of $\usg$ whose elements are isomorphic to $Q$ as a directed graph.
The second statement in Proposition \ref{prop:covers} implies that each tree module for $S_g$ is associated to a unique $F_g$ orbit in $\usg$.
This gives the following equation:
\begin{equation}\label{eq:qsgcount}
TM_{S_g}(d) = \sum_{Q\in \mathscr{Q}} [Q: \usg]\ TM_Q(d) .
\end{equation}

We now analyze the terms $ [Q: \usg]$.  For a given tree quiver $Q$, let $W_Q(k)$ be the number of surjective functions
\[
Q_1 \xto{\phi} \{1, 2, \dotsc, k\}
\]
satisfying the following condition:
\begin{align*}
\text{(W)} \quad \phi(\alpha) \neq \phi(\beta) \ \text{whenever }&\text{two arrows $\alpha, \beta$ share a common starting vertex,}\\
&\text{or share a common ending vertex.}
\end{align*}
One may think of these as labelings of the arrows of $Q$ which use precisely $k$ distinct labels and avoid the configurations
\[
\xleftarrow{i} \xto{i} \qquad \text{and} \qquad \xto{i} \xleftarrow{i}.
\]
Note that the set of functions $Q_1 \xto{\phi} \{1, 2, \dotsc, k\}$ is in bijection with morphisms of quivers from $Q$ to $S_k$, since the all vertices must be mapped to the unique vertex of $S_k$ and the compatibility conditions with the head and tail functions are automatically satisfied.  With this in mind, the condition (W) is equivalent to requiring that the morphism $f\colon Q \to S_k$ determined by the labeling $\phi$ is a \keyw{winding} in the sense of Krause \cite{MR1090218}.

\begin{prop}\label{prop:qsgcount}
For any tree quiver $Q$, we have that 
\begin{equation}\label{eq:qusgcount}
[Q: \usg] = \sum_{k=1}^{d-1} \frac{W_Q(k)}{\#\Aut(Q)}\binom{g}{k},
\end{equation}
where $d$ is the number of vertices of $Q$.

\end{prop}
\begin{proof}
We will show that each is equal to the cardinality of the orbit space $\frac{\Hom^\circ (Q, \usg)}{\Aut(Q)\times F_g}$, where $\Hom^\circ (Q, \usg)$ denotes the set of injective quiver morphisms from $Q$ to $\usg$. One equality is essentially by definition: two maps have the same image subquiver if and only if they are in the same $\Aut(Q)$ orbit, and then considering the $F_g$ action gives that
\[
\#\left(\frac{\Hom^\circ (Q, \usg)}{\Aut(Q) \times F_g}\right) = [Q: \usg] .
\]

On the other hand, any morphism in $f\in \Hom^\circ (Q, \usg)$ determines a map on arrow sets
\[
Q_1 \xto{\phi(f)} \{1, 2, \dotsc, g\},
\]
where an arrow $\za \in Q$ is sent to the index $i$ such that $f(\za)$ is labeled by $\za_i$ in $\usg$.  These automatically satisfy the condition (W) above by construction of $\usg$, which has precisely one incoming and one outgoing arrow labeled by each $\za_i$ at each vertex.  

Now we see that $\phi(f)$ completely determines the $F_g$ orbit of a quiver morphism $f$.  Indeed, we have that $\phi$ is constant on the $F_g$-orbits in $\Hom^\circ (Q, \usg)$ because the $F_g$ action does not change the label of an arrow in $\usg$, and whenever $\phi(f_1) \neq \phi(f_2)$ then there is an arrow $\za\in Q$ such that $f_1(\za)$ and $f_2(\za)$ have different labels, so $f_1$ and $f_2$ are not in the same $F_g$ orbit.

Now consider the sets
\[
P(I) := \setst{f\in \Hom^\circ (Q, \usg)}{\im \phi(f) = I} \quad \text{ for } I \subseteq \{1, 2, \dotsc, g\}.
\]
The above paragraph implies that the number of $F_g$ orbits of $P(I)$ is precisely the number of distinct $\phi(f)$ for $f \in P(I)$, which is just $W_Q(\# I)$.  Because $\Hom^\circ (Q, \usg)$ consists of injective morphisms, $\Aut(Q)$ acts freely on $P(I)$, so we can count the number of $\Aut(Q) \times F_g$ orbits in $P(I)$ to be 
\begin{equation}\label{eq:autqfg orbits}
\# \left(\frac{P(I)}{\Aut(Q) \times F_g} \right) = \frac{W_Q(\# I)}{\#\Aut(Q)}
\end{equation}

Now the partition
\[
\Hom^\circ (Q, \usg) = \coprod_{I \subseteq \{1, 2, \dotsc, g\}} P(I)
\]
combined with \eqref{eq:autqfg orbits} (and the fact $W_Q(k) = 0$ when $k > d-1$, the number of arrows of $Q$) gives that 
\[
\#\left(\frac{\Hom^\circ (Q, \usg)}{\Aut(Q) \times F_g}\right) = \sum_{k=1}^{d-1} \frac{W_Q(k)}{\#\Aut(Q)}\binom{g}{k},
\]
as desired.
\end{proof}

\subsection{The leading term of $TM_{S_g}(d)$}

We are now ready to prove the main theorem, which we state slightly differently than in the introduction.

\begin{theorem}
The number of tree modules of dimension $d$ for $S_g$ is a polynomial in $g$ of degree $d-1$.
Furthermore, when written in the basis $\setst{\binom{g}{i}}{0 \leq i \leq d-1}$ of $\bigoplus_{i=0}^{d-1}\Q g^i$, the polynomial $TM_{S_g}(d)$ has the same leading term as $A_{S_g}(d, 1)$, namely
\begin{equation}
\frac{2^{d-1}\#\mathscr{T}_d}{d}\binom{g}{d-1} = \sum_{Q \in \mathscr{Q}_d} \frac{(d-1)!}{\#\Aut Q} \binom{g}{d-1}.
\end{equation}
\end{theorem}
\begin{proof}
If $Q$ has more than $d$ vertices, then $TM_Q(d) = 0$ since $Q$ cannot have a sincere representation of dimension $d$.  So the sum in equation \eqref{eq:qsgcount} is finite, and combining this with \eqref{eq:qusgcount} immediately gives that $TM_{S_g}(d)$ is polynomial in $g$.  These equations also show that the highest degree term comes when $Q$ has $d$ vertices and $k = d-1$ in \eqref{eq:qusgcount}.  This proves the first statement.

A tree quiver with $d$ vertices admits only one sincere indecomposable of dimension $d$, namely the tree module $\mathbb{I}_Q$, so we have $TM_Q(d) = 1$ for $Q \in \mathscr{Q}_d$. 
Such a $Q$ has $d-1$ edges and so $W_Q(d-1) = (d-1)!$ since there is no chance to violate condition (W).  Thus the leading term of $TM_{S_g}(d)$ is precisely
\begin{equation}
\sum_{Q\in \mathscr{Q}_d} \frac{(d-1)!}{\#\Aut(Q)}\binom{g}{d-1}.
\end{equation}

The result \eqref{eq:RVlead} of Helleloid and Rodriguez-Villegas, expressed in our preferred basis, is that the leading term of $A_{S_g}(d,1)$ is $\frac{2^{d-1}\#\mathscr{T}_d}{d}\binom{g}{d-1}$, so now we need to show that the coefficients are equal.  Rearranging, we find that this is equivalent to showing that 
\begin{equation}
2^{d-1}\#\mathscr{T}_d = \sum_{Q\in \mathscr{Q}_d} \frac{d!}{\#\Aut(Q)} .
\end{equation}

Now it is straightforward to see that both expressions count tree quivers with labeled vertices.  On the left hand side, each tree with labeled vertices has $2^{d-1}$ orientations; on the right hand side, any unlabeled tree quiver with $d$ vertices admits $d!$ labelings of its vertices, with two labelings being indistinguishable if and only if they differ by an automorphism of $Q$.
\end{proof}

\section{Examples and Questions}

\subsection{Beyond leading terms}
For small $d$, we can compute $TM_{S_g}(d)$ from \eqref{eq:qsgcount} by explicitly enumerating all tree quivers up to $d$ vertices and calculating $[Q:\usg]$ and $TM_Q(d)$ for each.  To give the reader a feel for this, we illustrate the case $d=3$, listing each tree quiver $Q$ with 3 vertices and the corresponding value of $[Q:\usg]$ beneath it.
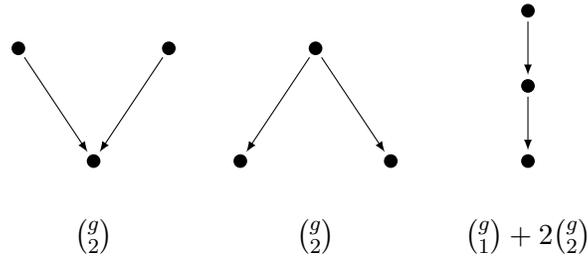
\[
\begin{tikzpicture}[point/.style={shape=circle,fill=black,scale=.5pt,outer sep=3pt},>=latex]
   \node[point] (1) at (-1,1) {};
   \node[point] (2) at (1,1) {};
  \node[point] (3) at (0,-.5) {};
    \node[draw, color=white,label={below: $\binom{g}{2}$}]at (0,-1) {};
  \path[->]
  	(1) edge (3) 
  	(2) edge (3);
   \end{tikzpicture}
\qquad
\begin{tikzpicture}[point/.style={shape=circle,fill=black,scale=.5pt,outer sep=3pt},>=latex]
   \node[point] (1) at (0,1) {};
   \node[point] (2) at (1,-.5) {};
  \node[point] (3) at (-1,-.5) {};
    \node[draw, color=white,label={below:$\binom{g}{2}$}]at (0,-1) {};
  
  \path[->]
  	(1) edge (2) 
  	(1) edge (3);
   \end{tikzpicture}
\qquad%
\begin{tikzpicture}[point/.style={shape=circle,fill=black,scale=.5pt,outer sep=3pt},>=latex]
   \node[point] (1) at (0,1.5) {};
   \node[point] (2) at (0,0.5) {};
  \node[point] (3) at (0,-.5) {};
    \node[draw, color=white,label={below:$\binom{g}{1} + 2\binom{g}{2}$}]at (0,-1) {};
  
  \path[->]
  	(1) edge (2) 
  	(2) edge (3);
   \end{tikzpicture}
\]
The case $d=4$ involves more 8 tree quivers, an additional 27 for $d=5$, and there are 92 distinct tree quivers with 6 vertices.
For this reason we have omitted the details of the computation, but the results are given here:
\begin{center}
\begin{tabular}{|c | c |}
\hline 
$d$ & $\#$ tree modules \\
\hline
1 & 1\\
2 & $\binom{g}{1}$\\
3 & $4\binom{g}{2} + \binom{g}{1}$\\
4 & $32\binom{g}{3} + 20\binom{g}{2} + \binom{g}{1}$ \\
5 & $400\binom{g}{4} + 428\binom{g}{3} + 93\binom{g}{2} +\binom{g}{1}$\\
6 & $6912\binom{g}{5}+10656\binom{g}{4} + 4524\binom{g}{3} + 448\binom{g}{2} +\binom{g}{1}$\\
\hline
\end{tabular}
\end{center}

Comparing with the table of values for $A_{S_g}(d,1)$ in the introduction of \cite{MR2543630}, we find that $TM_{S_g}(d) = A_{S_g}(d, 1)$ for $d \leq 5$ (as polynomials in $g$).  When $d=6$, however, the situation becomes more complicated because this is the first occurrence of terms with $TM_Q(d) \neq 1$.  Specifically, if $Q$ is of type $\widetilde{D}_4$ and $\delta$ is the minimal isotropic root, there are 6 tree modules of this dimension vector, which are precisely those that lie in exceptional tubes.
We find that
\begin{equation}
\begin{split}
TM_{S_g}(6) - A_{S_g}(6, 1) &= 16\binom{g}{4} + 12\binom{g}{3} + \binom{g}{2}\\
&= \frac{2}{3}g\left(g-\frac{1}{2}\right)\left(g-1\right)\left(g-\frac{3}{2}\right),
\end{split}
\end{equation}
so in particular $TM_{S_g}(6) \geq A_{S_g}(6, 1)$ for all $g \in \Z_{\geq 0}$ with strict inequality for $g > 1$.

Since the number of tree quivers with $d$ vertices grows rapidly with $d$, it quickly becomes difficult to enumerate all tree quivers up to $d$ vertices, calculate their automorphism groups, and find $W_Q(k)$ for each $1 \leq k \leq d-1$.  The case $d=7$ seems to be already intractable to do by hand, but perhaps suitable for a computer-aided calculation.

For an arbitrary quiver $Q$, we can refine our count of tree modules to be by dimension \emph{vector} $\dd$; we denote this number by $TM_Q(\dd)$.  We pose the following question:

\begin{question}\label{qu:ineq}
Let $Q$ be an arbitrary quiver and $\dd$ a dimension vector for $Q$.  Is it always true that $TM_Q(\dd) \geq A_Q(\dd, 1)$?
\end{question}

This inequality has been observed to be true in many examples; before summarizing a list of general situations under which it is true, we state a proposition which may be already known to experts.
\begin{prop}\label{prop:affineexcep}
Let $Q$ be an affine Dynkin quiver and $M$ a regular indecomposable lying in an exceptional tube.  Then $M$ is a tree module.
\end{prop}
\begin{proof}
Proceed by induction on the quasi-length of $M$.  If $M$ is quasi-simple in an exceptional tube, then $M$ is exceptional and known to be a tree module by Ringel's work \cite{Ringel:1998gf}.  If the quasi-length of $M$ is at least two, there exists a short exact sequence $0 \to M' \to M \to M'' \to 0$ with $M''$ quasi-simple and $\dim\Ext^1_Q(M'', M') = 1$.  Since $M'$ has smaller quasi-length than $M$, by induction it is a tree module and the fact that $\dim\Ext^1_Q(M'', M') = 1$ allows one to ``glue'' the tree structures of $M'$ and $M''$ to a tree structure for $M$.  (See \cite[Lemma~3.11]{Weist:2010fk} for the precise statement, or the final remark of \cite{Ringel:2010fk}.)
\end{proof}

In fact, this gluing procedure as one moves up the tube is remarkably easy to carry out in practice: once one decides how to glue adjacent modules in the mouth of the tube (the choice is not unique), this determines an algorithm to move up the tube and get an explicit tree structure for each module of larger quasi-length.

The answer to Question \ref{qu:ineq} is ``yes'' in the following cases:
\begin{itemize}
\item $\dd$ is a real Schur root of any $Q$: there is a unique indecomposable, by Kac's Theorem, and it is a tree module since it is exceptional, by Ringel's work.
\item $\dd$ is an imaginary root of an affine Dynkin $Q$: let $n+1$ be the number of vertices of $Q$.  Then $A_Q(\dd, q) = q+n$, which can be seen directly from matrix presentations of the indecomposables, so $A_Q(\dd, 1) = n+1$.

The loop quiver and Kronecker have no exceptional tubes but can be handled easily, so suppose $Q$ is not one of these.  Then the number of indecomposables of dimension $\dd$ in exceptional tubes is always at least $n+1$ (see for example the tables of \cite[\S6]{MR0447344} for one orientation, and use APR-tilting to get any other orientation).  By Proposition \ref{prop:affineexcep}, each of these is a tree module.
\item $\dd$ is a real non-Schur root of an affine Dynkin $Q$: there is a unique representation of dimension vector $\dd$, and it lies in an exceptional tube, so by Proposition \ref{prop:affineexcep} it is a tree module.
\end{itemize}

This question is related to a long-standing question about tree modules.  If Kac's conjecture that $A_Q(\dd,q)$ has nonnegative coefficients is true, then $A_Q(\dd, 1) \geq 1$ whenever $Q$ has an indecomposable representation of dimension vector $\dd$, with strict inequality when $\dd$ is an imaginary root.  Assuming this, Question \ref{qu:ineq} would generalize the following question \cite[p.~472]{Ringel:1998gf}.

\begin{question}[Ringel]
Let $Q$ be a quiver and $\dd$ be a dimension vector for which there exists an indecomposable representation. Does there always exists a tree module of dimension vector $\dd$, and more than one when $\dd$ is imaginary?
\end{question}

Ringel's question has been answered affirmatively by T. Weist for all roots of generalized Kronecker quivers \cite{Ringel:2010fk,MR2578596}, as well as isotropic roots and imaginary Schur roots of arbitrary quivers \cite{Weist:2010fk}.

\subsection{Some intriguing examples}\label{sect:examples}
As mentioned in Section \ref{sect:countpoly}, Kac conjectured that the coefficients of $A_Q(\dd, q)$ are always nonnegative.  He continues with an even more bold conjecture, which we now describe.  First he considers a space of isomorphism classes of indecomposable $\C$-representations of dimension vector $\dd$, obtained via successive applications of Rosenlicht's theorem \cite[\S5.3]{MR1648601}.  Let us call this space $\ind(Q, \dd)$, though care must taken because this procedure is not well-defined: there is not a \emph{unique} maximal open set admitting a geometric quotient (see \cite[\S6]{MR1648601} for an example).  So the isomorphism class of $\ind(Q, \dd)$ is not even well-defined as a variety.

For the sake of discussion, let us take any statement about ``$\ind(Q,\dd)$'' to be that there exists some choice of Rosenlicht quotients making the statement true for the resulting $\ind(Q,\dd)$.  Then Kac conjectures that we have a finite decomposition into constructible subsets which are affine algebraic cells 
\[
\ind(Q, \dd) = \coprod_{i=1}^N C_i, \qquad C_i \simeq \mathbb{A}_{\C}^{n_i}
\]
with the coefficient of $q^k$ in $A_Q(\dd, q)$ counting the number of cells of dimension $k$.  Notice that whenever such a decomposition exists, $A_Q(\dd, 1)$ counts the total number of cells; in general, $A_Q(\dd, 1)$ gives the Euler characteristic of $\ind(Q,\dd)$ in singular cohomology with compact support \cite[Lemma~8.1]{Reineke:2008fk}.  So we would like to know if there is any relation between tree modules and these cells.

\begin{example}
Consider the dimension vector $\dd=(2,2,1)$ for the quiver
\[
Q=
\vcenter{\hbox{
\begin{tikzpicture}[point/.style={shape=circle,fill=black,scale=.5pt,outer sep=3pt},>=latex]
  \node[point] (2) at (0,0) {};
  \node[point] (3) at (2,0) {};
  \node[point] (1) at (4,0) {};  
  \path[->]
  	(3.145) edge node[above] {$\za$} (2.35) 
  	(3.-145) edge node[below] {$\zb$} (2.-35)
  	(1) edge node[above] {$\zg$} (3);
   \end{tikzpicture}}} 
\]
of wild representation type.
Then by linear algebra or more sophisticated methods, it can be computed that
\begin{equation}
A_Q(\dd,q) = q^2 + 2q + 2 ,
\end{equation}
(here 2 is a ``bad'' characteristic). So we have that $A_Q(\dd,1) = 5$, and it is easy to check that $TM_Q(\dd) =5$ by {\it ad hoc} methods.
We do indeed get a decomposition of $\ind(Q, \dd)$ into one 2-cell, two 1-cells, and two 0-cells as predicted by Kac, 
and furthermore we can choose this decomposition so that the representation in the ``center'' of each cell (taking the parameters equal to 0) is a tree module.  
One such choice is given here:

\begin{align}\label{eq:bigcellexample}
\vcenter{\hbox{
\begin{tikzpicture}[point/.style={shape=circle,fill=black,scale=.5pt,outer sep=3pt},>=latex]
  \node[point] (2a) at (0,1) {};
  \node[point] (2b) at (0,-1) {};
  \node[point] (3a) at (2,1) {};
  \node[point] (3b) at (2,-1) {};
  \node[point] (1) at (4,0) {};  
  \path[->]
  	(3a) edge node[above] {$\zb$} (2a) 
  	(3b) edge node[below] {$\zb$} (2b)
  	(3a) edge node[below] {$\za$} (2b)
  	(1) edge node[above] {$\zg$} (3a);
   \end{tikzpicture}}} \qquad
&\vcenter{\hbox{
\begin{tikzpicture}[point/.style={shape=circle,fill=black,scale=.5pt,outer sep=3pt},>=latex]
  \node (2) at (0,0) {$\C^2$};
  \node (3) at (2,0) {$\C^2$};
  \node (1) at (4,0) {$\C$};
  \path[->]
  	(3.165) edge node[above] {$\twobytwo{\lambda}{\mu}{1}{0}$} (2.15) 
  	(3.-165) edge node[below] {$\twobytwo{1}{0}{0}{1}$} (2.-15)
  	(1) edge node[above] {$\twobyone{1}{0}$} (3);
   \end{tikzpicture}}} 
 \qquad (\lambda, \mu) \in \C^2\\
\vcenter{\hbox{
\begin{tikzpicture}[point/.style={shape=circle,fill=black,scale=.5pt,outer sep=3pt},>=latex]
  \node[point] (2a) at (0,1) {};
  \node[point] (2b) at (0,-1) {};
  \node[point] (3a) at (2,1) {};
  \node[point] (3b) at (2,-1) {};
  \node[point] (1) at (4,0) {};  
  \path[->]
  	(3a) edge node[above] {$\zb$} (2a) 
  	(3b) edge node[below] {$\zb$} (2b)
  	(3a) edge node[below] {$\za$} (2b)
  	(1) edge node[above] {$\zg$} (3b);
   \end{tikzpicture}}} \qquad
&\vcenter{\hbox{
\begin{tikzpicture}[point/.style={shape=circle,fill=black,scale=.5pt,outer sep=3pt},>=latex]
  \node (2) at (0,0) {$\C^2$};
  \node (3) at (2,0) {$\C^2$};
  \node (1) at (4,0) {$\C$};  
  \path[->]
  	(3.165) edge node[above] {$\twobytwo{\lambda}{0}{1}{\lambda}$} (2.15) 
  	(3.-165) edge node[below] {$\twobytwo{1}{0}{0}{1}$} (2.-15)
  	(1) edge node[above] {$\twobyone{0}{1}$} (3);
   \end{tikzpicture}}} 
 \qquad \lambda \in \C \\
\vcenter{\hbox{
\begin{tikzpicture}[point/.style={shape=circle,fill=black,scale=.5pt,outer sep=3pt},>=latex]
  \node[point] (2a) at (0,1) {};
  \node[point] (2b) at (0,-1) {};
  \node[point] (3a) at (2,1) {};
  \node[point] (3b) at (2,-1) {};
  \node[point] (1) at (4,0) {};  
  \path[->]
  	(3a) edge node[above] {$\zb$} (2a) 
  	(3b) edge node[below] {$\za$} (2b)
  	(1) edge node[above] {$\zg$} (3a)
  	(1) edge node[above] {$\zg$} (3b);
   \end{tikzpicture}}} \qquad
&\vcenter{\hbox{
\begin{tikzpicture}[point/.style={shape=circle,fill=black,scale=.5pt,outer sep=3pt},>=latex]
  \node (2) at (0,0) {$\C^2$};
  \node (3) at (2,0) {$\C^2$};
  \node (1) at (4,0) {$\C$};
  \path[->]
  	(3.165) edge node[above] {$\twobytwo{\lambda}{0}{0}{1}$} (2.15) 
  	(3.-165) edge node[below] {$\twobytwo{1}{0}{0}{0}$} (2.-15)
  	(1) edge node[above] {$\twobyone{1}{1}$} (3);
   \end{tikzpicture}}} 
 \qquad \lambda \in \C \\
\vcenter{\hbox{
\begin{tikzpicture}[point/.style={shape=circle,fill=black,scale=.5pt,outer sep=3pt},>=latex]
  \node[point] (2a) at (0,1) {};
  \node[point] (2b) at (0,-1) {};
  \node[point] (3a) at (2,1) {};
  \node[point] (3b) at (2,-1) {};
  \node[point] (1) at (4,0) {};  
  \path[->]
  	(3a) edge node[above] {$\za$} (2a) 
  	(3b) edge node[below] {$\za$} (2b)
  	(3a) edge node[below] {$\zb$} (2b)
  	(1) edge node[above] {$\zg$} (3b);
   \end{tikzpicture}}} \qquad
&\vcenter{\hbox{
\begin{tikzpicture}[point/.style={shape=circle,fill=black,scale=.5pt,outer sep=3pt},>=latex]
  \node (2) at (0,0) {$\C^2$};
  \node (3) at (2,0) {$\C^2$};
  \node (1) at (4,0) {$\C$};
  \path[->]
  	(3.165) edge node[above] {$\twobytwo{1}{0}{0}{1}$} (2.15) 
  	(3.-165) edge node[below] {$\twobytwo{0}{0}{1}{0}$} (2.-15)
  	(1) edge node[above] {$\twobyone{0}{1}$} (3);
   \end{tikzpicture}}} \\
\vcenter{\hbox{
\begin{tikzpicture}[point/.style={shape=circle,fill=black,scale=.5pt,outer sep=3pt},>=latex]
  \node[point] (2a) at (0,1) {};
  \node[point] (2b) at (0,-1) {};
  \node[point] (3a) at (2,1) {};
  \node[point] (3b) at (2,-1) {};
  \node[point] (1) at (4,0) {};  
  \path[->]
  	(3a) edge node[above] {$\za$} (2a) 
  	(3b) edge node[below] {$\za$} (2b)
  	(3a) edge node[below] {$\zb$} (2b)
  	(1) edge node[above] {$\zg$} (3a);
   \end{tikzpicture}}} \qquad
&\vcenter{\hbox{
\begin{tikzpicture}[point/.style={shape=circle,fill=black,scale=.5pt,outer sep=3pt},>=latex]
  \node (2) at (0,0) {$\C^2$};
  \node (3) at (2,0) {$\C^2$};
  \node (1) at (4,0) {$\C$};
  \path[->]
  	(3.165) edge node[above] {$\twobytwo{1}{0}{0}{1}$} (2.15) 
  	(3.-165) edge node[below] {$\twobytwo{0}{0}{1}{0}$} (2.-15)
  	(1) edge node[above] {$\twobyone{1}{0}$} (3);
   \end{tikzpicture}}} 
\end{align}
\end{example}

\begin{example}\label{eg:s2}
A description of $\ind(S_2, 2)$ is worked out carefully in \cite[\S6]{LeBruyn1986d} and also in \cite[\S6.2]{MR1648601}. We have $A_{S_2}(2,q) = q^5 + q^3$ so expect two cells of indecomposable representations.  After reducing to the case where the trace of each endomorphism is 0 by removing an affine factor of $\mathbb{A}^2$, Lebruyn gives the decomposition of $\ind(S_2, 2)$ into cells:
\begin{equation}
C_1 = \left\{\text{simple representations}\right\}\coprod
 \left\{\vcenter{\hbox{
\begin{tikzpicture}
		 \node (3) at (0,0) {$\C^2$} edge[in=135,out=205,loop] node[left] {$\twobytwo{a}{1}{0}{-a}$} () edge[in=45,out=-25,loop] node[right] {$\twobytwo{b}{1}{0}{-b}$} ();
		 \end{tikzpicture} }},  \quad a, b \in \C \right\}
\end{equation}
\begin{equation}
C_2 = 
 \left\{\vcenter{\hbox{
\begin{tikzpicture}
		 \node (3) at (0,0) {$\C^2$} edge[in=135,out=205,loop] node[left] {$\twobytwo{0}{a}{0}{0}$} () edge[in=45,out=-25,loop] node[right] {$\twobytwo{0}{1-a}{0}{0}$} ();
		 \end{tikzpicture} }},  \quad a \in \C \right\} .
\end{equation}
In this case there are two tree modules but both lie in the same cell $C_2$, at $a=0$ and $a=1$.  If we instead use the orbit representatives
\begin{equation}
C'_1 = \left\{\text{simple representations}\right\}\coprod
 \left\{\vcenter{\hbox{
\begin{tikzpicture}
		 \node (3) at (0,0) {$\C^2$} edge[in=135,out=205,loop] node[left] {$\twobytwo{a}{1}{0}{-a}$} () edge[in=45,out=-25,loop] node[right] {$\twobytwo{b}{a}{0}{-b}$} ();
		 \end{tikzpicture} }},  \quad a, b \in \C \right\}
\end{equation}
\begin{equation}
C'_2 = 
 \left\{\vcenter{\hbox{
\begin{tikzpicture}
		 \node (3) at (0,0) {$\C^2$} edge[in=135,out=205,loop] node[left] {$\twobytwo{0}{a}{0}{0}$} () edge[in=45,out=-25,loop] node[right] {$\twobytwo{0}{1}{0}{0}$} ();
		 \end{tikzpicture} }},  \quad a \in \C \right\} ,
\end{equation}
we get a tree module at the center of each cell, like the previous example.
\end{example}

It would be interesting to know under what conditions on $Q$ and $\dd$ one may find a cell decomposition of (some choice of) $\ind(Q,\dd)$ with a tree module in each cell.  It is straightforward to do this whenever $Q$ is affine Dynkin type and $\dd$ is an imaginary root, using explicit matrix presentations.  Of course the dream is to find a torus action on (some choice of) $\ind(Q,\dd)$ for which tree modules are fixed points.

\subsection*{Acknowledgments}  I am indebted to Ralf Schiffler for mentorship, encouragement, and many good conversations about representation theory while I pursued these ideas.  I would also like to thank Ed Green for explaining his work on group graded algebras to me, and explaining how to put Proposition \ref{prop:covers} in the appropriate generality here.  Anonymous referees also provided helpful feedback, and the open-source software system SAGE was instrumental in exploratory calculations.

\bibliographystyle{kinser-alpha}
\bibliography{ryanbiblio}

\end{document}